\documentclass{amsart}

\usepackage{epsfig}
\usepackage{amsthm,amsfonts}
\usepackage{amssymb,graphicx,color}
\usepackage[all]{xy}
\usepackage{verbatim}
\usepackage{hyperref}
\usepackage{enumitem}
\usepackage [latin1]{inputenc}

\newtheorem{theorem}{Theorem}[section]
\newtheorem*{theorem*}{Theorem}

\newtheorem{dfn}[theorem]{Definition}

\newtheorem{remark}[theorem]{Remark}

\newcommand{\De}{\mathcal D}
\newcommand{\G}{\Gamma}
\newcommand{\E}{\mathcal E}
\newcommand{\n}{\nabla}
\newcommand{\om}{\omega}
\newcommand{\w}{\wedge}
\newcommand{\p}{\partial}

\begin{document}

\title[Calabi type surfaces]
{Calabi type K\"ahler surfaces}

\author[W.Jelonek]{W\l odzimierz Jelonek}

{\address[W{\l}odzimierz Jelonek]{ Katedra Matematyki Stosowanej, Politechnika Krakowska, Warszawska 24, 31-155 Krak\'ow, Poland. \newline
              E-mail: {\tt wjelon@pk.edu.pl}}


\begin{abstract}
We show that the K\"ahler surface with opposite Hermitian, conformally  K\"ahler structure defined by a foliation of complex curves, is a QCH  K\"ahler surface of Calabi type. \end{abstract}

\maketitle
%\tableofcontents
\section{Introduction}  {\it  QCH K\"ahler surfaces}   (i.e. K\"ahler surfaces with quasi constant holomorphic sectional curvature) (see \cite{ganchev}, \cite{jel1}-\cite{jel3}, \cite{jel5}-\cite{jel7} \cite{jel-mul}, \cite{mul}) are K\"ahler surfaces  $(M, g, J)$
admitting a global, $2$-dimen\-sional, $J$-invariant distribution
$\De$ having the following property: The holomorphic curvature
$$K(\pi)=R(X, J X, J X, X)$$ of any $J$-invariant $2$-plane $\pi\subset
T_xM$, where $X\in \pi$ and $g(X, X)=1$, depends only on the point
$x$ and  number $|X_{\De}|=\sqrt{g(X_{\De},X_{\De})}$, where
$X_{\De}=p_{\De}X$ is the orthogonal projection of $X$ onto $\De$  (\cite{jel1}).  Every  QCH K\"ahler surface admits an opposite almost Hermitian structure $I$, (i.e. its K\"ahler form $\om_I$ is anti-selfdual), such that the Ricci tensor $\rho$ of $(M,g,J)$ is $I$-invariant and $(M,g,I)$ satisfies the second Gray condition  (\cite{gray}, p.605):

\begin{equation*}(G2)\\\\\\\\\ R(X, Y, Z, W)-R(IX, IY, Z, W)=\end{equation*}\begin{equation*}R(IX, Y, IZ, W)+R(IX, Y, Z, IW).\end{equation*}
 If $R$ is the curvature tensor of a
QCH K\"ahler manifold $(M, g, J)$, then there exist functions $a,b,c\in
C^{\infty}(M)$ such that
\begin{equation}R=a\Pi+b\Phi+c\Psi,\end{equation} where $\Pi$ is the
standard K\"ahler tensor of constant holomorphic curvature i.e.
\begin{equation}
\Pi(X,Y,Z,U)=\frac14(g(Y,Z)g(X,U)-g(X,Z)g(Y,U)\end{equation}
\begin{equation*}+g(JY,Z)g(JX,U)-g(JX,Z)g(JY,U)-2g(JX,Y)g(JZ,U)),\end{equation*}
the tensor $\Phi$ is defined by the following relation

\begin{equation} \Phi(X,Y,Z,U)=\frac18(g(Y,Z)h(X,U)-g(X,Z)h(Y,U)\end{equation}
\begin{equation*}+g(X,U)h(Y,Z)-g(Y,U)h(X,Z)
+g(JY,Z)h(JX,U)\end{equation*}\begin{equation*}-g(JX,Z)h(JY,U)+g(JX,U)h(JY,Z)-g(JY,U)h(JX,Z)\end{equation*}
\begin{equation*}-2g(JX,Y)h(JZ,U)-2g(JZ,U)h(JX,Y))\end{equation*}and finally
\begin{equation}\Psi(X,Y,Z,U)=-h(JX,Y)h(JZ,U)=-(h_J\otimes h_J)(X,Y,Z,U).\end{equation} where $h_J(X,Y)=h(JX,Y)$ and   $h=g\circ(p_{\De}\times p_{\De})$.
In the paper \cite{jel3} we have proven that if a K\"ahler surface $(M,g,J)$ admits a complex totally geodesic and conformal foliation and   $I$ is the Hermitian structure determined by this foliation and  $d\om_I=2\theta\w\om_I$  then $(M,g,J)$  is QCH if and only if  $d\theta^-=0$.  Here we give another proof  for the case  $d\theta=0$.
In this paper we study  K\"ahler surfaces $(M,g,J)$ with an opposite Hermitian non-K\"ahler structure  $I$.  It means that  $\om_I^2=-\om_J^2$ where  $\om_I,\om_J$ are K\"ahler forms of  $(M,g,I)$ and $(M,g,J)$ respectively.    These two structures define two complex distributions $\De_{\pm}=\ker(I\circ J\pm id)$ on  $M$.    We here will consider the case where one of the distributions $\De_{\pm}$ is inegrable.  Then  the integrable disribution $\De$ is conformal  and totally geodesic.  In fact $L_Vg =\theta(V) g$ for $V\in \De$ on  $\De^{\perp}$ where  $d\om_I=2\theta\w\om_I$ and  $\theta$ is the Lee form of $(M,g,I)$. We will prove that in this case, under the condition that  the structure $I$  is conformally K\"ahler, the surface   $(M,g,J)$ is a QCH manifold of Calabi type. We also show that   $(M,g,J)$  carry a special  K\"ahler-Ricci potential.  A K\"ahler surface $(M,g,J)$  with the non-vanishing  holomorphic Killing vector field  $X$  is called of Calabi type if the distribution    $\De=span\{X,JX\}$  defines the Hermitian, conformally K\"ahler  structure $I$.  In this sense we prove that  a K\"ahler surface  with   opposite non-K\"ahler  conformally K\"ahler structure $I$ is of Calabi type in the set  $U=\{x\in M:\theta_x\ne0\}$ where $\theta$ is the Lee form  of the structure $I$ which means that $d\om_I=2\theta\w\om_I$. A  K\"ahler surface $(M,g,J)$  is called ambi-K\"ahler if it admits  an opposite  conformally K\"ahler Hermitian structure $I$.  We  in fact could give an alternative definition of a Calabi surface:

\begin{dfn} An ambi-K\"ahler surface $(M,g,J)$ for which  $I$ is not K\"ahler is a Calabi type K\"ahler surface if one of the distributions $\De_{\pm}=\ker(I\circ J\pm id)$ is integrable.\end{dfn}

The four dimensional  K\"ahler and Hermitian surfaces are studied by many authors (see  ex.  \cite{der},  \cite{dunajski}, \cite{jel4}, \cite{web}).

\section{K\"ahler surfaces with an opposite Hermitian structure}  In our paper \cite{jel3}, we have proved that a K\"ahler surface admits an opposite Hermitian structure defined by a complex foliation $\De$, whose leafs are complex curves, if $\De$ is totally geodesic and conformal   or equivalently if it is holomorphic and totally geodesic.
Recall the definitions:
\begin{dfn} A foliation $\mathcal F$ on a Riemannian manifold
$(M,g)$ is called totally geodesic if for any vector fields $X,Y$ tangent to the
leaves od $\mathcal F$, we have  $\n_XY\in T\mathcal F$. \end{dfn}

\begin{dfn} A foliation $\mathcal F$ on a Riemannian manifold
$(M,g)$ is called conformal if, for any $V$ tangent to the
leaves od $\mathcal F$, the equation
$$L_Vg=\theta(V)g$$ holds on $T\mathcal F^{\perp}$, where $\theta$ is a one
form vanishing on $T\mathcal F^{\perp}$ . A foliation $\mathcal F$ is called homothetic if
it is conformal and $d\theta=0$.\end{dfn}
\medskip
\begin{dfn} A complex distribution $\De$ on a complex
manifold $(M,g,J)$ is called holomorphic if
$L_{\xi}J(TM)\subset\De$  for any $\xi\in\G(\De)$.\end{dfn}
\medskip
For a completness we recal the following theorem (see\cite{jel3}).  Note that here $d\om_I=2\theta\w\om_I$.
\begin{theorem} Let $(M,g,J)$ be a K\"ahler surface, dim M=4 and  let $\De,
dim\De=2$ be  a complex, conformal distribution  which means that
  $L_{\xi}g=\theta(\xi)g$ for $\xi\in\G(\De)$ on the distribution
$\E=\De^{\perp}$. Then for all $X,Y\in\G(\E)$ we have
$2\n_XY_{|\De}=-g(X,Y)\theta-\om(X,Y)J\theta$.\end{theorem}
\begin{proof}   We have  $g(\n_X\xi,Y)+g(X,\n_Y\xi)=\theta(\xi)g(X,Y)$.
Hence $g(\n_XY+\n_YX,\xi)=-\theta(\xi)g(X,Y)$. On the other hand
$\n_{JX}JX_{|\De} = -\frac12\theta g(X,Y)$ which implies $$\n_{JX}X_{|\De}=\frac12\om(X,JX)J\theta.$$ Note that
$2\n_XJX_{|\De}=-\frac12\om(X,JX)J\theta$ .  Consequently
$[X,JX]_{|\De}=\om(X,JX)J\theta$  and $[X,Y]_{|\De}=\om(X,Y)J\theta$ which means that
$2\n_XY_{|\De}=-\om(X,Y)J\theta-g(X,Y)\theta$ for $X,Y\in\E=\De^{\perp}$.\end{proof}

 \section{ AmbiK\"ahler K\"ahler surfaces with an opposite structure given by a foliation}   Let $(M,g,J)$ be a K\"ahler surface with a K\"ahler form  $\om(X,Y)=g(JX,Y)$ and   $\De$ be a totally geodesic and holomorphic foliation. Then  $\De$   defines an opposite Hermitian structure $I$  such that  \begin{equation}J_{|\De}=I_{|\De},  J_{|\De^{\perp}}=-I_{|\De^{\perp}}.\end{equation}  Let  $\om_I(X,Y)=g(IX,Y)$ be a K\"ahler form of $(M,g,I)$.  Then   $d\om_I=2\theta\w\om_I$  where $\theta$ is a Lee form of $(M,g,I)$. Next we shall assume that $\theta\ne0$ and    $\theta\in\De$, $d\theta=0$.    Locally  $\theta=du$ and $\n_{\n u}\n u\in\De$.  If  $H^1(M)=0$  then $u$ is defined globally.  We shall assume that  $\theta=du$.
\begin{theorem} A K\"ahler surface $(M,g,J)$ with  an opposite non-K\"ahler  conformally K\"ahler structure $I$ defined by an integrable complex distribution $\De$ is of Calabi type in the set  $U=\{x\in M:\theta_x\ne0\}$ where $\theta$ is the Lee form  of the structure $I$. The surface  $(M,g,J)$ admits a special K\"ahler-Ricii potential. \end{theorem}
   \begin{proof}  Since $\De$ is totally geodesic we have$H^u(\n u,X)=0$ for  $X\in\De^{\perp}$. Denote  $Q=g(\n u,\n u)$. Then for $X\in\De^{\perp}$  we get  $XQ=2g(\n_X\n u,\n u)=2H^u(\n u,X)=0$. Note that   $grad Q=2\n_{\n u}\n u\in \De$. Let us define the 2-form  $\om_2=\frac1Qdu\w Jdu=\frac1Qdu\w d^cu$.    Then   $\om=\om_1+\om_2$   and   $\om_I=\om_2-\om_1$.
Recall that    $$dd^cu=H^u(JX,Y)-H^u(JY,X)$$  and  consequently  $$dd^cu=\lambda\om_2+\mu\om_1.$$   Thus   $$0=d\lambda\w\om_2+\lambda d\om_2+d\mu\w\om_1+\mu d\om_1.$$    Hence  $d\mu\w du=0$ since  $d\om_1=-\theta\w\om_1, d\om_2= \theta\w\om_1$.  Now we prove that $\mu= Q$. In fact for $X,Y\in\De^{\perp}$
$$(\n_XY)_{|\De}=-\frac12g(X,Y)\n u -\frac12\om(X,Y)J\n u.$$  Hence  $$H^u(X,Y)=-g(\n u,\n_XY)=\frac12g(X,Y)Q.$$ From the equality $dQ\w du=0$ it follows that $dQ=2\lambda_1du$ which means that $Q'(u)=2\lambda_1$ and  $\n_{\n u}\n u=\lambda_1\n u$.    Note that also  $d\lambda_1\w du=0$.   Hence  $\n_{J\n u}\n u=\lambda_2J\n\ u$. We also have    $J\theta\w dJ\theta=\mu J\theta\w \om_1$. Consequently

\begin{equation} (d J\theta)^2= d\mu\w J\theta\w\om_1+\mu dJ\theta\w\om_1-\mu J\theta\theta\w\om_1.\end{equation}

Note that  $\mu=Q(u)$ and
\begin{equation} 2\lambda\mu\om_1\w\om_2= Q'(u)|\theta|^2vol+\mu\lambda vol-\mu|\theta|^2vol.\end{equation}

Hence    $\mu\lambda vol=(Q'(u)-\mu)|\theta|^2vol$.  Hence  \begin{equation}\lambda=\lambda_1+\lambda_2=\frac{(Q'(u)-\mu)|\theta|^2}{|\theta|^2}=Q'(u)-|\theta|^2=2\lambda_1-Q.\end{equation}

  Consequently $d\lambda\w du=0$ and $d\lambda_2\w du=0$,$\lambda_2=\lambda_1-Q$.    Now we find a function   $g=g(u)$ in such a way that the field   $\xi=g\n u$ would be the gradient of some function  $\tau$  and  $J\n \tau$ would be a holomorphic  Killing vector field.
We choose  $g$ in such a way    that  $[\xi,J\xi]=0$.
\begin{gather} [g\n u,gJ\n u]=g(\n ug)J\n u-g(J\n ug)\n u+g^2[\n u, J\n u]\\=g(\n ug)J\n u+g^2[\n u, J\n u].\end{gather}  Hence   $$ g'(u)|\n u|^2=g(\lambda_2-\lambda_1)=-Qg$$   and  $$(\ln g)'=-1.$$  Let  $g=e^{-u}$.
If  $u=\ln v$ then  $\theta=d\ln v$  and $g=\frac1v$.

Let   $\tau=-g(u)$.  Then   $d\tau=gdu$.   Now we prove  that  $H^{\tau}(X,Y)=H^{\tau}(JX,JY)$.   First  $\n\tau\in\De$, hence   $H^{\tau}(X,Y)=H^{\tau}(JX,JY)$ for  $X,Y\in\De^{\perp}$. Note that  for $\xi=\n\tau$  we have  $J\n_{\xi}\xi=\n_{J\xi}\xi$,
 $J(\n_{\xi}J\xi)=-\n_{\xi}\xi=\n_{J\xi}J\xi$.   Also  $H^{\tau}(\xi,X)=H^{\tau}(J\xi,JX)=0$  for  $X\in\De^{\perp}$.  It follows that  $\n\tau$ is a holomorphic vector field  and $J\n\tau$ is a holomorphic Killing vector field.

$\om_2=\frac1Qd\tau\w d^c\tau$.   Now    \begin{equation}  dd^c\tau= 2\lambda \om_2+2\mu\om_1\end{equation}  for some functions $\lambda,\mu$  on $M$.
Thus
\begin{equation} 0=d\lambda\w\om_2+d\mu\w\om_1+(\mu-\lambda)d\om_2,\end{equation} where  $\om=g(J.,.)=\om_1+\om_2$.  But  $-d\om_1=d\om_2=\theta\w\om_1$.  Hence  $d\lambda\w\om_2=0$ and

 \begin{equation} d\om_2=-\frac{d\mu}{(\lambda-\mu)}\w\om_1.\end{equation}

On the other hand    \begin{equation} d\om_2=\frac1Qd\tau\w dd^c\tau=-2\frac{\mu}{|\n\tau|^2}d\tau\w\om_1.\end{equation}
It implies that \begin{equation}\frac{d\mu}{\lambda-\mu}=\frac{2\mu d\tau}{|\n\tau|^2}\end{equation} and

  \begin{equation}  d\mu=\frac{2(\lambda-\mu)\mu}{|\n\tau|^2}d\tau=\Lambda d\tau.\end{equation}

  Let    $Q=|\n\tau|^2$. Then

  \begin{gather*} d(\frac Q{\mu})(X)=\frac{2g(\n_X\n\tau,\n\tau)}{\mu}-\frac Q{\mu^2}\Lambda d\tau\\=\frac{2\lambda}{\mu}d\tau-\frac{Q\Lambda}{\mu^2}d\tau=\\(\frac{2\lambda\mu-Q\Lambda}{\mu^2})d\tau=2d\tau.\end{gather*}

    Hence \begin{equation} \frac Q{\mu}=2(\tau-c),\end{equation}   where  $c\in\mathbb R$.  Since  $\tau$ is the potential of a holomorphic Killing vector field  it follows that  $d(\Delta \tau)=Ric(\n\tau,.)$.  On the other hand $\Delta\tau=2\lambda+2\mu$  and  $d\Delta\tau\w d\tau=0$.   Hence $\n\tau$ is an eigenfield of the Ricci tensor $Ric$.  It follows that $\tau$ is a special K\"ahler-Ricii potential on  K\"ahler surface $(M,g,J)$.  Recall that $\tau$ is a special K\"ahler-Ricci potential on a K\"ahler manifold $(M,g,J)$ if $\tau$ is a nonconstant Killing potential on $(M,g,J)$ and, at every point with $d\tau\ne0$, all nonzero tangent vectors orthogonal to $v=\n \tau$ and $u=Jv$ are eigenvectors of both $\n d\tau$ and the Ricci tensor $Ric$ (see [2]).
  It is also clear that $Ric$ is $I$-invariant, hence  $(M,g,J)$ is a Calabi type QCH  K\"ahler surface.\end{proof}
While proving Th.3.1 we have used:

\begin{theorem}Let  $f,u:M\rightarrow \mathbb R$ be a smooth functions on a manifold  $M$.   Let  $M'=\{x\in M: d_xu\ne0\}$ be a dense, connected and open set in $M$.  If  $df\w du=0$  then there exists a smooth function $g:\mathbb R\rightarrow\mathbb R$ such  that   $f=g\circ u$.\end{theorem}

\begin{proof} First we  consider $M'$.    Let  $d_xu\ne0$ and  $(x_1=u,x_2,..,x_n)$ be a local map around $x$.  In this coordinates  $f=g(x_1,x_2,..,x_n)$ and since $df\w du=0$ we have  $\p_if=0$ for $i=2,3,..,n$.  Hence
$g=g(u)=f$. If we have two such maps then   $g(u)=g_1(u)$ on the intersection,  hence we can glue all such functions $g$ and obtain  $f=g(u)$ in $M'$. We shall assume that $M'$ is connected. Then  $u(M')$ is an interval in $\mathbb R$. Note that if $x\in M'$  then $u(x)$ is in the interior of $u(M')$.  Hence   $u(M')=(\inf u,  \sup u)$.\end{proof}

\begin{remark}  Note that if a K\"ahler surface  $(M,g,J)$, such that  $H^1(M)=0$  admits a special K\"ahler-Ricci potential  $\tau$  with  $\mu\ne0$  and such that  $\tau\ne c$ everywhere, then   $(M,g,J)$ is ambi-K\"ahler QCH surface.  It admits an opposite conformally K\"ahler Hermitian structure $I$  such that on the set  $U=\{x:\theta_x\ne0\}$  it is of Calabi type.\end{remark}

\begin{proof}   If  $\tau$ is a special K\"ahler-Ricci potential on $(M,g,J)$ then  $\De =span\{\nabla\tau,$ $J\nabla \tau\}$ is a totally geodesic and holomorphic foliation.   Hence  the structure $I$ determined by $\De$ on the set $U=\{x:d\tau(x)\ne0\}$  is Hermitian. It is also clear that is locally  conformally K\"ahler and hence conformally K\"ahler since $H^1(M)=0$.  If   $\tau$ is everywhere different from $c$  then  the foliation $\De$ extends to the whole of $M$  (see \cite{der1}).  Hence $I$ is globally determined  and  $(M,g,J)$  is ambi-K\"ahler with an opposite structure $I$. It is clear that  $(M,g,J)$ is of Calabi type.\end{proof}

\end{document}